\documentclass[a4paper,12pt]{amsart}
\usepackage{amsmath,amssymb,amsfonts,amsthm,exscale,calc}
\usepackage[latin1]{inputenc}
\usepackage{latexsym,textcomp}
\usepackage{graphicx,epsf}
\usepackage{dsfont}
\usepackage[german, english]{babel}
\usepackage{ifthen}
\usepackage{tikz}
\usepackage{hyperref}

\newtheorem{lemma}{Lemma}[section]

\newtheorem{thm}[lemma]{Theorem}
\newtheorem{proposition}[lemma]{Proposition}
\newtheorem{corollary}[lemma]{Corollary}

\theoremstyle{definition}
\newtheorem{definition}[lemma]{Definition}
\newtheorem{example}[lemma]{Example}

\theoremstyle{remark}

\numberwithin{equation}{section}
\newcommand{\comment}[1]{}

\newcommand{\R}{{\mathbb R}}

\newcommand{\N}{{\mathbb N}}
\newcommand{\EE}{{\mathbb E}}
\newcommand{\PP}{{\mathbb P}}

\newcommand{\A}{{\mathcal{A}}}
\newcommand{\LL}{{\mathcal{L}}}

\newcommand{\QQ} {{\mathcal{Q}}}
\newcommand{\DD}{{\mathcal{D}}}
\newcommand{\FF}{{\mathcal{F}}}

\newcommand{\supp}{{\mathrm {supp}\,}}

\newcommand{\Deg}{{\mathrm {Deg}}}

\newcommand{\al}{{\alpha}}

\newcommand{\de}{{\delta}}
\newcommand{\eps}{{\varepsilon}}

\newcommand{\ph}{{\varphi}}
\newcommand{\lm}{{\lambda}}

\newcommand{\si}{{\sigma}}
\newcommand{\Om}{{\Omega}}

\newcommand{\ov}[1]{\overline{ #1}}

\newcommand{\Hm}[1]{\leavevmode{\marginpar{\tiny%
$\hbox to 0mm{\hspace*{-0.5mm}$\leftarrow$\hss}%
\vcenter{\vrule depth 0.1mm height 0.1mm width \the\marginparwidth}%
\hbox to 0mm{\hss$\rightarrow$\hspace*{-0.5mm}}$\\\relax\raggedright
#1}}}

\def\Deg{\mathrm{Deg}}

\renewcommand{\epsilon}{\varepsilon}
\renewcommand{\phi}{\varphi}

\begin{document}

\title[Intrinsic metrics on  graphs]{Intrinsic metrics on graphs: A survey}
\author[M. Keller]{Matthias Keller}
\address{Mathematisches Institut \\Friedrich Schiller Universit{\"a}t Jena \\07743 Jena, Germany } \email{m.keller@uni-jena.de}

\maketitle

\begin{abstract}
A few years ago various disparities for Laplacians on graphs and
manifolds were discovered. The corresponding results are mostly
related to volume growth in the context of unbounded geometry.
Indeed, these disparities can now be resolved by using so called
intrinsic metrics instead of the combinatorial graph distance. In
this article we give an introduction to this topic and survey recent
results in this direction. Specifically, we cover topics such as
Liouville type theorems for harmonic functions, essential
selfadjointness, stochastic completeness and upper escape rates.
Furthermore, we determine the spectrum as a set via solutions,
discuss upper and lower spectral bounds by isoperimetric constants
and volume growth and study $p$-independence of spectra under a
volume growth assumption.
\end{abstract}
\tableofcontents
\section{Introduction}
There are many parallels in the  analysis of Laplacians on graphs and manifolds. However, starting with examples discovered by Wojciechowski \cite{Woj1, Woj2,Woj3} various disparities surfaced that show very different behavior for certain phenomena on manifolds and graphs. These phenomena appeared in the context of unbounded geometry, i.e., when the corresponding graph Laplacian is an unbounded operator. In particular, these phenomena do not show for the normalized Laplacian. Furthermore, a common feature these phenomena share are that they deal with questions of volume growth and more specifically with distances.

Independently, Frank/Lenz/Wingert \cite{FLW} developed a framework to study so called intrinsic metrics for general, in particular non-local, regular Dirichlet forms. Such metrics had already proven to be very effective in the context of  strongly local regular Dirichlet forms, see e.g. the work of Sturm \cite{Stu}. In \cite{FLW} already various applications were given  for general Dirichlet forms.

Indeed, the Laplace-Beltrami operator on a Riemannian manifold arises from a strongly local Dirichlet form and the corresponding intrinsic metric turns out to be the Riemannian metric. On the other hand, Laplacians on graphs arise from Dirichlet forms that are non-local and, thus, their forms do not fit in the framework of \cite{Stu}. So, the question arises what is a suitable metric for a given graph. The most immediate choice is the combinatorial graph distance which is  given by the minimal number of edges needed to connect two vertices by a path.
It turns out that the combinatorial graph distance is equivalent to an intrinsic metric in the sense of  \cite{FLW} if and only if the geometry is bounded. As a consequence many results known for Riemannian manifolds can be proven for graphs in the case of bounded geometry, but they fail for graphs with unbounded geometry if one considers the combinatorial graph distance.

Now, \cite{FLW} provides a very general concept of intrinsic metrics which can be applied also to  graph Laplacians. And, indeed, these metrics work as a remedy to recover the results of Riemannian manifolds for graphs (even with unbounded geometry) in very great generality. From this perspective, the case of bounded geometry and, in particular, the normalized Laplacian appears as a special case.

In this survey we review the concept of intrinsic metrics and present recent results in the context of weighted graphs. In the next two sections we introduce the basic concepts of weighted graphs, Laplacians  and intrinsic metrics.

The sections after these introductional parts  are structured in the following way: We first introduce the question and discuss the known results in the manifold setting. Then we examine the situation for graphs with bounded geometry and illustrate the problems that arise for  unbounded geometry. Finally, we present results involving intrinsic metrics to recover the results for manifolds for general weighted graphs.

In Section~\ref{s:Liouville} we study harmonic functions and discuss Liouville type theorems involving $\ell^{p}$ growth bounds, $p\in(1,\infty)$, which are classical results of Yau \cite{Yau} and Karp \cite{Ka} for manifolds. As a corollary, we obtain a criterion for recurrence. The recent results for graphs in this section are mainly based on results found in \cite{HK}.

We use these results in Section~\ref{s:ESA} to draw consequences about the domain of the generators and, in particular, about essential selfadjointness for the case $p=2$ which is a result of Gaffney \cite{Ga} and Roelcke \cite{R} for complete manifolds. For graphs, these are results based on \cite{HKMW,HK}.

In Section~\ref{s:SC} we consider stochastic completeness, a property which is equivalent to uniqueness of solutions of the heat equation and the Poisson equation in $\ell^{\infty}$. For this phenomena Wojciechowski \cite{Woj3} discovered examples of  only little more than cubic volume growth with respect to the combinatorial graph distance that are not stochastically complete while the volume growth bound for manifolds is superexponential as proven by Grigor'yan, \cite{Gri86}. With the help of intrinsic metrics the  result from the manifold setting can be recovered for locally finite graphs, see \cite{Fol,Hu1}.
After stochastic completeness, we take a look at upper escape rates of the related Markov process in Section~\ref{s:UER}. This is  based on results of \cite{HS}.

The $\ell^{2}$-spectrum as a set can be determined by existence of subexponentially growing solutions in $\ell^{2}$ which is a classical result by Shnol' \cite{Sh} in $\R^{n}$ and was generalized by Boutet de Monvel/Lenz/Stollmann to strongly local Dirichlet forms, \cite{BLS}. This is discussed in Section~\ref{s:Shnol} and the references for graphs are found in \cite{FLW,K3}.

After this result on the spectrum as a set, we focus on the lower edge and put the focus on lower and upper spectral bounds via geometric invariants.

First, we consider lower bounds   via isoperimetric constants in Section~\ref{s:Cheeger}. The use of such a constant  to estimate the bottom of the spectrum goes back to a classical theorem of Cheeger, \cite{Ch}. For graphs a first result in this direction is due to Dodziuk \cite{Do} which is however only applicable for graphs with bounded vertex degree. This problem was pointed out in 1986 Dodziuk/Kendall, \cite{DK},  and they gave a proof for the normalized Laplacian recovering the original result of Cheeger in this particular situation.  Since then the general case remained open and a priori it is  not even clear  what role do metrics play in this problem at all. Recently, the problem was solved in \cite{BKW} by the use of intrinsic metrics that now enter the definition of the boundary measure in the isoperimetric constant.

Upper bounds on the bottom of the spectrum can be given in terms of exponential volume growth of distance balls in Section~\ref{s:Brooks} which is a result of Brooks \cite{Br} which is paralleled  by result of Sturm \cite{Stu}. Such a result  fails for unbounded graph Laplacians and the combinatorial graph distance, since there are graphs with only little more than cubic volume growth that admit a spectral gap, \cite{KLW}. Again, using intrinsic metrics the result that holds for Riemannian manifolds can be recovered for graphs and even for general Dirichlet forms, \cite{HKW}.

Finally, in Section~\ref{s:p_ind} we study the spectrum of the Laplacian in dependence of the  underlying $\ell^{p}$ space, $p\in[1,\infty]$. This question was answered for Schr\"odinger operators in $\R^{n}$ by Hempel/Voigt \cite{HV} in 1986. This question was later addressed by Sturm for manifolds \cite{Stu2} in terms of volume growth. It turns out that in the case of uniform subexponential volume growth the spectrum can be shown to be $p$-independent. For results on graphs the results are found in \cite{BHK}.

\section{Graphs and Laplacians}
In this section we introduce the basic notions for weighted
symmetric graphs. We mainly follow the framework developed in
\cite{KL1}.


\subsection{Graphs}
Let $X$ be a discrete and countably infinite space and  $m$ a measure of full support on $X$, that is a function $m:X\to(0,\infty)$ which is extended additively to sets via $m(A)=\sum_{x\in A}m(x)$, $A\subseteq X$.

A graph $(b,c)$ over $X$  is a symmetric function $b:X\times X\to[0,\infty)$ with zero diagonal and
\begin{align*}
    \sum_{y\in X}b(x,y)<\infty,\quad x\in X,
\end{align*}
and $c:X\to[0,\infty)$. We say two vertices $x,y\in X$ are \emph{neighbors}  if $b(x,y)>0$. In this case we write $x\sim y$.  The function $c$ can be thought to describe one-way-edges to a virtual point at infinity or as a potential or as a killing term.
If $c\equiv 0$, then we speak of $b$ as a graph over $X$.
 If we already fixed a measure, then we speak of graphs over $(X,m)$.

We say a graph is \emph{connected} if for all $x,y\in X$ there are $x=x_{0}\sim\ldots\sim x_{n}=y$. If a graph is not connected one can restrict the attention to the connected components. Therefore, we assume that the graphs under consideration are connected.

Given a pair $(b,c)$ an important special case of a measure $m$ is the \emph{normalizing measure}
\begin{align*}
    n(x)=\sum_{y\in X}b(x,y)+c(x),\quad x\in X.
\end{align*}
Another important special case is the \emph{counting measure} $m\equiv 1$.

We say a graph is \emph{locally finite} if every vertex has only finitely many neighbors, that is if the \emph{combinatorial vertex degree} $\deg$ is finite at every vertex
\begin{align*}
   \deg(x)= \#\{y\in X\mid x\sim y\} <\infty,\quad\mbox{ for all } x\in X.
\end{align*}

We speak of a graph with \emph{standard weights} if $b:X\times X\to\{0,1\}$ and $c\equiv0$. In this case the normalizing measure $n$ equals the {combinatorial vertex degree} $\deg$. Obviously, by summability of $b$ about vertices, graphs with standard weights are locally finite.

\subsection{Generalized forms and formal Laplacians}
We let $C(X)$ be the set of real valued functions on $X$ and $C_{c}(X)$ be the subspace of functions in $C(X)$ of finite support. For the corresponding theory for complex valued function the theory can be transferred by the virtue of standard arguments presented in \cite[Appendix B]{HKLW}.

For a graph $(b,c)$ over $X$, the generalized quadratic form $\QQ:C(X)\to[0,\infty]$ is given by
\begin{align*}
\QQ(f)=\frac{1}{2}\sum_{x,y\in X} b(x,y)|f(x)-f(y)|^{2}+\sum_{x\in X} c(x)|f(x)|^{2}
\end{align*}
with generalized domain
\begin{align*}
    \DD=\{f\in C(X)\mid \QQ(f)<\infty\}.
\end{align*}
Since $\QQ^{\frac{1}{2}}$ is a semi norm and satisfies the parallelogram identity, $\QQ$ yields a semi scalar product on $\DD$  by polarization via
\begin{align*}
\QQ(f,g)=\frac{1}{2}\sum_{x,y\in X} b(x,y)(f(x)-f(y)){(g(x)-g(y))}+\sum_{x\in X} c(x)f(x){g(x)}.
\end{align*}
Moreover, for functions in
\begin{align*}
    \FF=\{f\in C(X)\mid \sum_{y\in X}b(x,y)|f(y)|^{2}<\infty \mbox{ for all }x\in X\},
\end{align*}
we define the generalized Laplacian $\LL:\FF\to C(X)$ by
\begin{align*}
    \LL f(x)=\frac{1}{m(x)}\sum_{y\in X}b(x,y)(f(x)-f(y)) + \frac{c(x)}{m(x)}f(x).
\end{align*}
In \cite[Lemma 4.7]{HK}, a \emph{Green's formula} is shown for functions $f\in \FF$ and $\ph\in C_c(X)$
\begin{align*}
    \frac{1}{2}\sum_{x,y\in X} b(x,y)(f(x)-f(y)){(\ph(x)-\ph(y))}+\sum_{x\in X} c(x)f(x){\ph(x)}
\end{align*}
\begin{align*}
    =\sum_{x\in X}\LL f(x){\ph(x)}m(x)=\sum_{x\in X} f(x){\LL\ph(x)}m(x).
\end{align*}
In the special case when the functions under consideration are in a certain Hilbert space the formula above might be expressed by scalar products, see Section~\ref{s:MarkovianSG}.


We call $f\in\FF$ a \emph{solution} (respectively \emph{subsolution} or \emph{supersolution}) for $\lm\in\R$ if $(\LL-\lm)f=0$ (respectively  $(\LL-\lm)f\le0$ or  $(\LL-\lm)f\ge0$). A solution (respectively subsolution or supersolution) for $\lm=0$ is called a \emph{harmonic} (respectively \emph{subharmonic} or \emph{superharmonic}).

We say a function $f\in C(X)$ is \emph{positive} if $f$ is non-trivial and $f\ge0$  and \emph{strictly positive} if $f>0$.

For two functions $f,g\in C(X)$ we denote the minimum of $f$ and $g$ by $f\wedge g$ and the maximum of $f$ and $g$ by $f\vee g$.

A \emph{Riesz space} is a linear space equipped with a partial
ordering that is consistent with addition, scalar multiplication and
where the maximum and the minimum of two functions exist.

An important well known fact which is needed in the subsequent is that in order to study existence of non-constant (respectively non-zero) solutions for $\lm\leq 0$ in a Riesz space, it suffices to study positive subharmonic functions.

\begin{lemma}\label{l:SolutionsVsSubsolutions} Let $(b,c)$ be a connected graph over $(X,m)$ and $\FF_{0}\subseteq \FF$ be a Riesz space. If there are no non-constant  positive subharmonic functions in $\FF_{0}$, then there are no non-constant solutions for $\lm\leq 0$ in $\FF_{0}$. In particular, any solution for $\lm<0$ is zero.
\end{lemma}
\begin{proof}For a  solution $f$ for $\lm\leq0$ the positive part $f_{+}=f\vee 0$, the negative part $f_{-}=-f\vee 0$ and the modulus $|f|=f_{+}+f_{-}$ can directly be seen to be non-negative subharmonic functions. Thus, the statement follows from connectivity. The 'in particular' is obvious.
\end{proof}


\subsection{Dirichlet forms and their generators}
The form and Laplacian introduced above are defined on spaces with
rather few structure. Next, we will consider these objects
restricted to suitable Hilbert and Banach spaces.

Let $\ell^{p}(X,m)$ be the canonical real valued $\ell^{p}$-spaces, $p\in[1,\infty]$, with norms
\begin{align*}
    {\|f\|}_{p}=\Big(\sum_{x\in X}|f(x)|^{p}m(x)\Big)^{\frac{1}{p}}, \; p\in[1,\infty),
\end{align*}
\begin{align*}
    {\|f\|}_{\infty}=\sup_{x\in X}|f(x)|.
\end{align*}
As $\ell^{\infty}(X,m)$ does not depend on $m$ we also write $\ell^{\infty}(X)$.
For $p=2$, we have a Hilbert space $\ell^{2}(X,m)$ with scalar product
\begin{align*}
    \langle{f,g}\rangle= \sum_{x\in X}f(x){g(x)}m(x),\qquad f,g\in \ell^{2}(X,m),
\end{align*}
and we denote the norm  $\|\cdot\|={\|\cdot\|}_{2}$.


\subsubsection{Dirichlet forms}
Restricting the form $\QQ$ to $\DD\cap\ell^{2}(X,m)$, we see by Fatou's lemma that this restriction is lower semi-continuous and, thus, closed. Hence, the restriction of $\QQ$ to $C_{c}(X)$ is closable.

Let $Q=Q_{b,c}$ be the quadratic form given by
\begin{align*}
    D(Q)&=\overline{C_{c}(X)}^{\|\cdot\|_{\QQ}},\; \mbox{ where } \|\cdot\|_{\QQ} =\big(\QQ(\cdot)+\|\cdot\|^{2}\big)^{\frac{1}{2}}\\    Q(f)&=\frac{1}{2}\sum_{x,y\in X}b(x,y)|f(x)-f(y)|^{2}+\sum_{x\in X}c(x)|f(x)|^{2},\qquad f\in D(Q).
\end{align*}
It can be seen that $Q$ is a  \emph{Dirichlet form}  that is for any $f\in D(Q)$ we have $0\vee f\wedge 1\in D(Q)$ and
\begin{align*}
    Q(0\vee f\wedge 1)\leq Q(f).
\end{align*}
(see \cite[Theorem~3.1.1]{Fuk} for general theory and for a proof in the graph setting see \cite[Proposition 2.10]{Schm}).
Obviously, $Q$ is \emph{regular}, that is  $C_{c}(X)\cap D(Q)$ is dense in $D(Q)$ with respect to ${\|\cdot\|}_{\QQ}$ and dense in $C_{c}(X)$ with respect to ${\|\cdot\|}_{\infty}$.

As it turns out, by \cite[Theorem~7]{KL1},  all regular Dirichlet
forms on discrete spaces are given in this way. Indeed, this can be
also derived directly from the Beurling-Deny representation formula
\cite[Theorem~3.2.1 and Theorem~5.2.1]{Fuk}.

\begin{thm}[Theorem~7 in \cite{KL1}] If $q$ is a regular Dirichlet on $\ell^{2}(X,m)$, then there is a graph $(b,c)$ such that $q=Q_{b,c}$.
\end{thm}

\subsubsection{Markovian semigroups and their generators}\label{s:MarkovianSG}
By general theory (see e.g. \cite[Satz 4.14]{Weidmann}), $Q$
yields a positive selfadjoint operator $L$ with domain $D(L)$ viz
\begin{align*}
    Q(f,g)=\langle L^{\frac{1}{2}}f,L^{\frac{1}{2}}g\rangle, \qquad f,g\in D(Q).
\end{align*}
By Green's formula it can be seen directly that $L$ is a restriction of $\LL$ which reads in the case when $f\in D(Q)\cap \mathcal{F}$ and $\ph\in C_{c}(X)$ as
\begin{align*}
    Q(f,\ph)=\langle \mathcal{L}f,\ph\rangle,
\end{align*}
where the right hand side also equals $\langle f,\mathcal{L}\ph\rangle$ whenever $\mathcal{L}\ph\in\ell^{2}(X,m)$.

By the second Beurling-Deny criterion $L$ gives rise to
a Markovian semigroup $e^{-tL}$, $t>0$, which extends consistently to all $\ell^{p}(X,m)$, $p\in[1,\infty]$, and is strongly continuous for $p\in[1,\infty)$. \emph{Markovian} means that for functions $0\leq f\leq 1$, one has $0\leq e^{-tL}f\leq 1$.\\

We denote the generators of $e^{-tL}$ on $\ell^{p}(X,m)$, $p\in[1,\infty)$, by $L_{p}$, that is
\begin{align*}
    D(L_{p})=\big\{f\in \ell^{p}(X,m)\mid g=\lim_{t\to0}\frac{1}{t}(I-e^{-tL})f\mbox{ exists in } \ell^{p}(X,m)\big\}
\end{align*}
\begin{align*}    L_{p}f=g
\end{align*}
and $L_{\infty}$ is defined as the adjoint of $L_{1}$.

It can be shown that  $L_{p}$ is a restriction of $\LL$.

\begin{thm}[Theorem~5 in \cite{KL1}]
Let $(b,c)$ be a  graph over $(X,m)$ and $p\in[1,\infty]$. Then,
\begin{align*}
    L_{p}f=\LL f,\qquad f\in D(L_{p}).
\end{align*}
\end{thm}


\subsubsection{Boundedness of the operators}\label{s:DF:Boundedness}
We next comment on the boundedness of the form $Q$ and the operator $L$. The  theorem below is taken from \cite{HKLW} and an earlier version can be found in \cite[Theorem~11]{KL2}.

\begin{thm}[Theorem~9.3 in \cite{HKLW}]\label{l:DF:Boundedness} Let $(b,c)$ be a graph over $(X,m)$. Then the following are equivalent:
\begin{itemize}
  \item [(i)] $X\to[0,\infty)$, $x\mapsto\frac{1}{m(x)}\Big(\sum_{y\in X}b(x,y)+c(x)\Big)$ is a bounded function.
  \item [(ii)] $Q$ is bounded on $\ell^{2}(X,m)$.
  \item [(iii)] $L_{p}$ is bounded for some $p\in[1,\infty]$.
  \item [(iv)] $L_{p}$ is bounded for all $p\in[1,\infty]$.
\end{itemize}
Specifically, if the function in (i) is bounded by $D<\infty$, then $Q\leq 2D$ and ${\|L\|}_{p}\leq 2D$, $p\in[1,\infty]$.
\end{thm}


\subsubsection{The compactly supported functions as a core}
It shall be observed that $C_{c}(X)$ is in general not included in $D(L)$. Indeed, one can give a characterization for this situation. The proof is rather immediate and we refer to \cite[Proposition~3.3]{KL1} or \cite[Lemma 2.7.]{GKS} for a reference.

\begin{lemma}\label{l:DF:C_c} Let $(b,c)$ be a graph over $(X,m)$. Then the following are equivalent:
\begin{itemize}
  \item [(i)]  $C_{c}(X)\subseteq D(L)$.
  \item [(ii)] $\LL C_{c}(X)\subseteq \ell^{2}(X,m)$
  \item [(iii)] The functions $X\to[0,\infty)$, $y\mapsto\frac{1}{m(y)}b(x,y)$ are in $\ell^{2}(X,m)$ for all $x\in X$.
\end{itemize}
In particular, the above assumptions are satisfied if the graph is locally finite or
\begin{align*}
    \inf_{y\sim x}m(y)>0, \qquad x\in X.
\end{align*}
Moreover, the above assumptions imply $\ell^{2}(X,m)\subseteq \FF$.
\end{lemma}
\begin{proof} The equivalence of (ii) and (iii) follows from the abstract definition of the domain of $L$ as $D(L)=\{f\in D(Q)\mid \mbox{there is }l\in \ell^{2}(X,m)\mbox{ such that }\langle l,\ph\rangle=Q(f,\ph)\mbox{ for all }\ph\in D(Q)\}$.
The equivalence of (i) and (ii) is a direct calculation, see \cite[Proposition~3.3]{KL1} and the 'in particular' statements are also immediate, see \cite{KL1,GKS}.
\end{proof}


\subsubsection{Graphs with standard weights}
Two important  special cases are graphs with \emph{standard
weights}, that is $b:X\times X\to\{0,1\}$ and $c\equiv 0$.

For the counting measure $m\equiv 1$, we denote  the operator $L$ on
$\ell^{2}(X)=\ell^{2}(X,1)$ by $\Delta$. It acts as
\begin{align*}
    \Delta f(x)=\sum_{y\sim x}(f(x)-f(y)),\quad f\in D(\Delta),\,x\in X.
\end{align*}
By Lemma~\ref{l:DF:Boundedness} the operator $\Delta$ is bounded if
and only if the combinatorial vertex degree $\deg$ is bounded, i.e.,
$\sup_{x\in X}\deg(x)<\infty$. In the unbounded case one has still
$C_{c}(X)\subseteq D(\Delta)$ by Lemma~\ref{l:DF:C_c} since standard
weights imply local finiteness.

For the normalizing measure $n=\deg$, we call the operator $L$ on
$\ell^{2}(X,\deg)$ the \emph{normalized Laplacian} and denote it by
$\Delta_n$. The operator $\Delta_{n}$  acts as
\begin{align*}
    \Delta_{n}f(x)=\frac{1}{\deg(x)}\sum_{y\sim x}(f(x)-f(y)), \quad f\in\ell^{2}(X,\deg),\,x\in X,
\end{align*}
and, also by Lemma~\ref{l:DF:Boundedness}, $\Delta_{n}$ is always bounded by $2$.


\section{Intrinsic metrics}
In this subsection we discuss the notion of intrinsic metrics for
weighted graphs.  Such metrics have proven to be very effective in
the context of strongly local Dirichlet forms, see e.g. \cite{Stu}.
Recently, this concept was generalized to all regular Dirichlet
forms and  systematically studied by Frank/Lenz/Wingert in
\cite{FLW}, (see also \cite{MU} for an earlier mentioning of the
criterion for certain non-local forms). We will demonstrate in this
article that these metrics can be used to obtain analogous  results
for graphs as in the case of manifolds.


\subsection{Definition of intrinsic metrics}
We call a symmetric map $\rho:X\times X\to[0,\infty)$ with zero diagonal a \emph{pseudo metric} if it satisfies the triangle inequality. By \cite[Lemma 4.7, Theorem~7.3]{FLW} it can be seen that the definition below coincides with the definition of an intrinsic metric for general regular Dirichlet forms, \cite[Definition~4.1]{FLW}.

\begin{definition}A pseudo metric $\rho$ is called an \emph{intrinsic metric} with respect to a graph $(b,c)$ over $(X,m)$ if for all $x\in X$
\begin{align*}
    \sum_{y\in X}b(x,y)\rho^{2}(x,y)\leq m(x).
\end{align*}
\end{definition}

Similar definitions of such metrics were also introduced later  in the context of  graphs or jump processes under the name adapted metrics, see \cite{Fol,Fol2,GHM,Hu,HS,MU}.


\subsection{Examples and relations to other metrics}

\subsubsection{The degree path metric}

A  specific example of an intrinsic metric was introduced by Huang, \cite[Definition 1.6.4]{Hu} and it also appeared in \cite{Fol0}.
Consider the pseudo metric $\rho_{0}:X\times X\to[0,\infty)$ given by
\begin{align*}
    \rho_{0}(x,y)=\inf_{x=x_{0}\sim x_{1}\sim\ldots\sim x_{n}=y}\sum_{i=1}^{n}\Big(\mathrm{Deg}(x_{i-1})\vee \mathrm{Deg}(x_{i})\Big)^{-\frac{1}{2}}, \qquad x\in X,
\end{align*}
where $\mathrm{Deg}:X\to(0,\infty)$ is the  \emph{weighted vertex
degree}  given by
\begin{align*}
    \mathrm{Deg}(x)=\frac{1}{m(x)}\sum_{y\in X}b(x,y),\qquad x\in X.
\end{align*}
We call such a metric that minimizes sums of weights over paths of edges a \emph{path metric}.

It can be seen directly that $\rho_{0}$ defines an intrinsic metric for the graph $(b,c)$ over $(X,m)$
\begin{align*}
    \sum_{y\in X}b(x,y)\rho_{0}^{2}(x,y)&\leq \sum_{y\in X}\frac{b(x,y)}{\mathrm{Deg}(x)\vee \mathrm{Deg}(y)}
    \leq  \sum_{y\in X}\frac{b(x,y)}{\mathrm{Deg}(x)}=m(x).
\end{align*}

There is the following intuition behind the definition of $\rho_{0}$. Consider the Markov process $(X_{t})_{t\ge0}$ associated to the semigroup $e^{-tL}$ via
\begin{align*}
    e^{-tL}f(x)=\EE_{x}(f(X_{t})),\qquad x\in X,
\end{align*}
where $\EE_{x}$ is the expected value conditioned on the process starting at $x$. The random walker modeled by this process jumps from a vertex $x$ to a neighbor $y$ with probability $b(x,y)/\sum_{z}b(x,z)$. Moreover, the probability of not having left $x$ at time $t$ is given by \begin{align*}
    \PP_{x} (X_{s}=x, 0\leq s\leq t)=e^{-\Deg(x)t}.
\end{align*}
Qualitatively, this indicates that the larger $\Deg(x)$, the faster the random walker leaves $x$.

Looking at the definition of $\rho_{0}(x,y)$, the larger the degree of either $x$ or $y$ the closer are the two vertices. Combining these two observations, we see that the faster the random walker jumps along an edge the shorter the edge is with respect to $\rho_{0}$. Of course, the jumping time along an edge connecting $x$ to $y$ is not symmetric and depends on whether one jumps from $x$ to $y$ or from $y$ to $x$ as  the degrees of $x$ and $y$ can be very different. In order to get a symmetric function, $\rho_{0}$ favors the vertex with the larger degree and the faster jumping time.

There is a direct analogy to the Riemannian setting in terms of mean
exit times of small balls. Consider a small ball $B_{r}$ of radius
$r$ on a $d$-dimensional Riemannian manifold. The first order term
of the mean exit time of $B_{r}$ is $r^{2}/2d$, \cite{P}.

On a locally finite graph for a vertex $x$ a 'small' ball with respect to $\rho_{0}$ can  be thought to have radius $r=\inf_{y\sim x}\rho(x,y)/2$, namely this ball contains only the vertex itself. Now, computing the mean exit time of this ball gives $1/\Deg(x)\ge r^{2}$, where equality holds whenever $\Deg(x)=\max_{y\sim x}\Deg(y)$.


\subsubsection{The combinatorial graph distance}\label{s:Metric:CGD}

We call the path metric defined by
\begin{align*}
d(x,y)&=\\
&\min\#\{n\in\N_{0}\mid \mbox{there are } x_{0},\ldots,x_{n}\mbox{ with } x=x_{0}\sim \ldots\sim x_{n}=y\}
\end{align*}
the \emph{combinatorial graph distance}. The next lemma shows that the combinatorial graph distance is equivalent to an intrinsic metric if and only if the graphs has bounded geometry, i.e., $\Deg$ is bounded. This was already  noted in \cite{FLW,KLSW}.

\begin{lemma}\label{l:intrinsic}
Let $b$ be a graph over $(X,m)$. The  following are equivalent:
\begin{itemize}
  \item [(i)] The combinatorial graph distance $d$ is equivalent to an intrinsic metric.
  \item [(ii)] $\Deg$ is a bounded function.
  \item [(iii)] $L$ is  a bounded operator.
\end{itemize}
\end{lemma}
\begin{proof}(i)$\Rightarrow$(ii):
Let $\rho$ be an intrinsic metric such that $C^{-1}\rho\leq d\leq  C\rho$. Then,
\begin{align*}
\sum_{x\in X}b(x,y)= \sum_{x\in X}b(x,y)d^{2}(x,y)\leq C^{2}\sum_{x\in X}b(x,y)\rho^{2}(x,y)\leq C^{2} m(x),
\end{align*}
for all $x\in X$. Hence, $\Deg\leq C^{2}$. \\
(ii)$\Rightarrow$(i): Assume $\Deg\leq C$ and consider the degree path metric $\rho_{0}$. Then, $\rho_{1}=\rho_{0}\wedge 1$ is an intrinsic metric as well. Clearly, $\rho_{1}\leq d$. On the other hand, by $\Deg\leq C$ we immediately get
 $\rho_{1}\ge C^{-\frac{1}{2}}d.$\\
The equivalence (ii)$\Leftrightarrow$(iii) follows from Theorem~\ref{l:DF:Boundedness}.
\end{proof}

The theorem implies in particular that in the case of the normalizing measure $m=n$ the combinatorial graph distance is an intrinsic distance as $\Deg=n/m=1$ in the case of $c\equiv0$ and $\Deg\leq n/m=1$ in general.

Furthermore, for   a graph with standard weights and the counting measure associated to the Laplacian $\Delta$, the combinatorial graph distance $d$  is a multiple of an intrinsic metric  if and only if the combinatorial vertex degree $\deg$ is bounded since $\Deg=\deg$.


\subsubsection{Comparison to the strongly local case}
An important difference to the case of strongly local Dirichlet
forms is that in the graph case there is no maximal intrinsic
metric. For example for a complete Riemannian manifold $M$ the
Riemannian distance ${d_{M}}$ is  the maximal $C^{1}$ metric $\rho$
that  satisfies
\begin{align*}
|\nabla_{M} \rho(o,\cdot)|\leq 1,
\end{align*}
for all  $o\in M$, where $\nabla_{M}$ is the Riemannian gradient. In
fact, $d_{M}$ can be recovered by the formula
\begin{align*}
    d_{M}(x,y)=\sup\{f(x)-f(y)\mid f\in C_{c}^{\infty}(M)\; |\nabla_{M}f|\leq 1\},\quad x,y\in M.
\end{align*}

Now, for discrete spaces the maximum of two intrinsic metrics is not necessarily an intrinsic metric. In particular, consider the pseudo-metric $\si$
\begin{align*}
    \si(x,y)=\sup\{f(x)-f(y)\mid f\in\mathcal{A}\},\qquad x,y\in X,
\end{align*}
where
\begin{align*}
    \mathcal{A}= \big\{f:X\to\R\mid \sum_{y\in X}b(x,y)|f(x)-f(y)|^{2}\leq m(x)\mbox{ for all }x\in X\big\}.
\end{align*}
As discussed for the Riemannian case above, the strongly local analogue $d_{M}$ of $\si$ defines the maximal intrinsic metric in the strongly local case, but $\si$ is in general not even equivalent to an intrinsic metric in the graph case.

A basic example where this can be seen immediately can be found in \cite[Example 6.2]{FLW}. More generally, this can be seen for arbitrary tree graphs with standard weights and the counting measure to which the operator $\Delta$ is associated. In this case $\si=\frac{1}{2}d$ and by discussion above we already know that the combinatorial graph distance $d$ is typically not an intrinsic metric for $\Delta$.

An abstract way to see that $\si$ is in general not intrinsic is
discussed in \cite[Section 1]{KLSW}. Namely, the set of
$1$-Lipschitz continuous functions $\mathrm{Lip}_{\rho}$ with
respect to an intrinsic metric $\rho$  with Lipschitz constant one
is included in $\A$ and $\mathrm{Lip}_{\rho}$ is closed under taking
suprema. On the other hand,  $\mathcal{A}$ is in general not closed
under taking suprema. Hence, in general $\mathrm{Lip}_{\rho}$ is not
equal to $\mathcal{A}$.

\subsubsection{Resistance metrics}
Another important metric appears in the context of resistance metrics. Let $r:X\times X\to[0,\infty)$ be given by
\begin{align*}
    r(x,y)=\sup\{f(x)-f(y)\mid f\in \DD,\;\QQ(f)\leq 1\},\qquad x,y\in X.
\end{align*}
Indeed, $r$ is the square root of the resistance metric as it appears e.g. in \cite{LP}, see \cite[Theorem~3.20]{GHKLW}. In \cite{GHKLW} this metric is related to intrinsic metrics.

\begin{thm}[Theorem~3.14 in \cite{GHKLW}] Let $b$ be a connected graph over $X$.
Then, for all $x,y\in X$
\begin{align*}
     r(x,y)=
     \sup\{\rho(x,y) \mid
     \mbox{intrinsic metric for $b$ over
     $(X,m)$ with $m(X)= 1$}\}.
\end{align*}
\end{thm}


\subsubsection{Another path metric}\label{s:Metric:CdV}
Colin de Verdi\`ere/Torki-Hamza/Truc \cite{CdVTHT} studied a path
pseudo metric $\de$ which is given as
\begin{align*}
    \de(x,y)=\inf_{x=x_{0}\sim\ldots\sim x_{n}=y}\sum_{i=0}^{n-1}\Big(\frac{m(x)\wedge m(y)}{b(x,y)}\Big)^{\frac{1}{2}}, \qquad x,y\in X.
\end{align*}
As discussed in \cite{HKMW}, this metric is equivalent to the
intrinsic metric $\rho_{0}$ if and only if the combinatorial vertex
degree $\deg$ is bounded on the graph.


\subsection{A Hopf-Rinow theorem}

We  stress that  in general an intrinsic metric $\rho $ (and in
particular $\rho_{0}$) is not a metric and $(X,\rho)$ might not even
be locally compact, as can be seen from examples in \cite[Example
A.5]{HKMW}. However, for locally finite graphs and path metrics such
as $\rho_{0}$ the situation is much tamer. For example one can show
a Hopf-Rinow type theorem which in parts can also be found in
\cite{Mi}. Recall that for a path $\gamma=(x_{n})$ the length with
respect to a metric $\rho$  is given by
$l(\gamma)=\sum_{j\ge0}\rho(x_{j},x_{j+1})$ and, moreover, a path
$\gamma=(x_{n})$ is called a \emph{geodesic} with respect to a
metric $\rho$ if $\rho(x_{0},x_{n})=l((x_{0},\ldots,x_{n}))$ for all
$n$.

\begin{thm}[Theorem~A1 in \cite{HKMW}]\label{t:HopfRinow} Let $(b,c)$ be a locally finite connected graph over $(X,m)$ and $\rho$ be a path metric. Then, the following are equivalent:
\begin{itemize}
  \item [(i)] $(X,\rho)$ is complete as a metric space.
  \item [(ii)] $(X,\rho)$ is geodesically complete, that is any infinite geodesic has infinite length.
  \item [(iii)] The  distance balls in $(X,\rho)$ are pre-compact (that is finite).
\end{itemize}
\end{thm}


\subsection{Some important conditions}
In the general situation, we will often make  assumptions which are discussed next.

We say a pseudo metric $\rho$ admits \emph{finite balls} if (iii) in
the theorem above is satisfied for $\rho$:
\begin{itemize}
  \item [(B)] The distance balls $B_{r}(x)=\{y\in X\mid \rho(x,y)\leq r\}$ are finite for all $x\in X$, $r\ge0$.
\end{itemize}

A somewhat weaker assumption is that the  \emph{degree is bounded on
balls}:
\begin{itemize}
  \item [(D)] The restriction of $\Deg$ to  $B_{r}(x)$ is bounded for all $x\in X$, $r\ge0$.
\end{itemize}
Clearly, (B) implies (D). Moreover, (D) is equivalent to the fact
that $\LL$ restricted to the $\ell^{2}$ space of any distance ball
is a bounded operator, confer Theorem~\ref{l:DF:Boundedness}.

The assumptions (B) and (D) can be understood as bounding $\rho$ in a certain sense from below. Next, we come to an assumption which may be understood as an upper bound.

We say a pseudo-metric $\rho$ has \emph{finite jump size} if
\begin{itemize}
  \item [(J)] The jump size $s=\sup\{\rho(x,y)\mid x,y\in X,\,x\sim y\}$ is finite.
\end{itemize}

The assumptions (B) and (J) combined have  strong geometric consequences.

\begin{lemma}Let $(b,c)$ be a graph over $(X,m)$ and $\rho$ be an pseudo metric. If $\rho$ satisfies (B) and (J), then the graph is locally finite.
\end{lemma}
\begin{proof} If there was a vertex with infinitely many neighbors, then there would be a distance ball containing all of them by finite jump size. However, this is impossible by (B).
\end{proof}
\subsection{Construction of cut-off functions}\label{s:Metric:estimate} From an analyst's point of view a major interest in metrics is to construct cut-off functions with desirable properties. Let us illustrate in how intrinsic metrics serve this purpose.

Given a intrinsic metric $\rho$, a subset $A\subseteq X$  and $R>0$, the most  basic cut-off function is defined by $\eta=\eta_{A,R}:X\to[0,\infty)$
\begin{align*}
    \eta(x)=\Big(1-\frac{\rho(x,A)}{R}\Big)\wedge 0, \qquad x\in X.
\end{align*}

Such  functions are often used to approximate a solution $f$ by $\eta_{B_{r},R}f$, where $B_{r}$ is a ball with respect to an intrinsic metric $\rho$ about some vertex.

Let us illustrate an estimate which often occurs  as a crucial step in the analysis. To this end, let $s\in[0,\infty]$ be the jump size of $\rho$. In the analysis one often arrives at a term such as below which is then further estimated using the intrinsic metric property as illustrated below
\begin{align*}
\lefteqn{\sum_{x\in X}|f(x)|^{2}\sum_{y\in X} b(x,y)(\eta(x)-\eta(y))^{2}}\\
&=\sum_{x\in B_{R+s}\setminus B_{r-s}}|f(x)|^{2}\sum_{y\in B_{R+s}\setminus B_{r-s}} b(x,y)(\eta(x)-\eta(y))^{2}\\
&\leq \frac{1}{(R-r)^{2}}\sum_{x\in B_{R+s}\setminus B_{r-s}} |f(x)|^{2}\sum_{y\in B_{R+s}\setminus B_{r-s}} b(x,y)(\rho(x,B_{r})-\rho(y,B_{r}))^{2}\\
&\leq\frac{1}{(R-r)^{2}}\sum_{x\in B_{R+s}\setminus B_{r-s}} |f(x)|^{2}\sum_{y\in B_{R+s}\setminus B_{r-s}} b(x,y)\rho^{2}(x,y)\\
&\leq\frac{1}{(R-r)^{2}}\sum_{x\in B_{R+s}\setminus B_{r-s}} {|f(x)|^{2}}m(x)\\
&=\frac{1}{(R-r)^{2}}\|f1_{B_{R+s}\setminus B_{r-s}} \|^{2}. \end{align*}

For example consider a graph with standard weights and the counting measure. If ones constructs $\eta$ with the combinatorial graph distance instead, this yields the vertex degree function $\deg(x)$ instead of $m(x)=1$ in the fourth step. Hence, in order to do the final step one has to assume additionally that $\deg$ is bounded.

Finally, let us discuss the virtue of the assumptions of finite balls (B) and finite jump size (J) in light of these considerations:

If (B) holds, then $\eta_{B_{r},R}f$ is in $C_{c}(X)$ which is for example useful to apply Green's formula.

Secondly, if (J) holds then $s<\infty$. So, at the end of the estimate above we only have the $\ell^{2}$ norm of $f$ on a ball rather than on the whole space.



\section{Liouville type theorems}\label{s:Liouville}

The classical Liouville theorem states that if a harmonic function
is bounded from above, then the function is constant. Here, we look
into boundedness assumptions such as $\ell^{p}$ growth bounds.
First, we present Yau's $L^{p}$ Liouville theorem and Karp's
improved bound for manifolds. Secondly, we discuss the case of the
normalized Laplacian for graphs and the results that have been
proven for this operator. Thirdly, we present theorems that recover
Yau's and Karp's results for general weighted graphs using intrinsic
metrics. As a consequence, this yields  a sufficient criterion for
recurrence. Finally, we round off the section by a result which
seems to hold exclusively for discrete spaces.

Throughout this section, keep in mind the fact that absence of non-constant positive subharmonic functions implies the absence of non-constant harmonic functions, Lemma~\ref{l:SolutionsVsSubsolutions}.


\subsection{Yau's and Karp's theorem for manifolds}
We consider a connected Riemannian manifold $M$, together with its Laplace Beltrami operator $\Delta_{M}$. A twice continuously differentiable function $f$ on $M$ is called harmonic (respectively subharmonic) if $\Delta_{M}f=0$ (respectively $\Delta_{M}f\leq0$).

In 1976 Yau \cite{Yau} proved that on a complete Riemannian manifold $M$ any harmonic function or positive subharmonic function in $L^{p}(M)$ is already constant.

This result was later strengthened by Karp in 1982, \cite{Ka}.
Namely, any harmonic function or positive subharmonic function $f$
that  satisfies
\begin{align*}
    \inf_{r_{0}>0}\int_{r_{0}}^{\infty} \frac{r}{{\|f1_{B_{r}\|}}_{p}^{p}}dr=\infty,
\end{align*}
is already constant, where $1_{B_{r}}$ is the characteristic function of the geodesic ball $B_{r}$ about some arbitrary point in the manifold.
Karp's result has Yau's theorem as a direct consequence.

Later in 1994 Sturm generalized Karp's theorem to  strongly local Dirichlet forms, where balls are taken with respect to the intrinsic metric. The underlying assumption on  the metric is that it generates the original topology and all balls are relatively compact.


\subsection{Liouville theorems for normalized graph Laplacians} For graphs $b$ over $(X,m)$ the first results in this direction were obtained for the normalizing measure $m=n$. In this case, the operator $L$ is bounded, see Section~\ref{l:DF:Boundedness}, and the combinatorial graph distance $d$ is an intrinsic metric, see Section~\ref{s:Metric:CGD}. (Of course, the fact that a function is harmonic depends only on the graph $b$ and not on the measure $m$, but being in an $\ell^{p}$ space does.)

Starting 1997 with Holopainen/Soardi \cite{HoS},  Rigoli/Salvatori/Vignati \cite{RSV}, Masamune \cite{M}, eventually  in 2013
Hua/Jost \cite{HJ} showed that  if a harmonic or positive subharmonic function $f$ satisfies
\begin{align*}
    \liminf_{r\to\infty}\frac{1}{r^2}\|f1_{B_{r}(x)}\|_{p}^{p}<\infty,
\end{align*}
for some $p\in (1,\infty)$ and $x\in X$, then $f$ must be constant. Here,  the balls are taken with respect to the natural graph distance. This  directly implies Yau's theorem for $p\in(1,\infty)$. Moreover, Hua/Jost \cite{HJ} also show Yau's theorem for $p=1$.

\subsection{Liouville theorems involving intrinsic metrics}
For graphs $b$ over a general discrete measure space one can  not expect such results to hold without any further conditions: Namely, being harmonic does not depend on the measure, so for any harmonic function $f$ there is a measure $m$ such that $f$ is in $\ell^{p}(X,m)$, $p\in (1,\infty)$.

The following theorem for $p\in(1,\infty)$  is found in \cite[Corollary~1.2]{HuK} with the additional assumption of finite jump size (J). Below we sketch  how to omit (J) based on an idea which was communicated by Huang \cite{Hu2}.
For $p=2$, the  theorem (without the assumption (J)) is found in  \cite{GKS} based on ideas developed in \cite{HuK,HKMW} and \cite{Mi}, see also \cite{Mi2,MT}.

\begin{thm}
[Corollary~1.2 in \cite{HuK} and \cite{Hu2}]\label{t:Yau} Let
$(b,c)$ be a connected graph over $(X,m)$ and $\rho$ be an intrinsic
metric  with bounded degree on balls (D). If $f\in\ell^{p}(X,m)$,
$p\in(1,\infty)$, is a positive subharmonic function then $f$ is
constant.
\end{thm}
\begin{proof}[Sketch of the proof] The key element of the proof is a Caccioppoli inequality stating that  there is $C>0$ such that such that for any positive subharmonic function $f$ and $0<r<R$
\begin{align*}
\sum_{x,y\in B_r}b(x,y)(f(x)\vee f(y))^{p-2}
|f(x)-f(y)|^{2}&\leq \frac{C}{(R-r)^{2}}\|f\|_{p}^{p}.
\end{align*}
Letting first $R\to\infty$ and afterwards $r\to\infty$ readily yields constancy of $f$.\\
The assumption (J) was used in \cite[Lemma 3.1]{HuK} only for   Green's formula applied to $f\in \FF$ such that  $f_{r}=f1_{B_{r}}\in \ell^{p}(X,m)$ and $g\in \ell^{q}(X,m)$,  $\supp g\subseteq B_{r}$ with $1/p+1/q=1$. It was pointed out by  Huang \cite{Hu2} that  (D) is sufficient to show
\begin{align*}
    \sum_{x\in X}(\LL f_{r})(x)g(x)m(x) = \frac{1}{2}\sum_{x,y\in B_r}b(x,y)(f_{r}(x)-f_{r}(y)) {(g(x)-g(y))},
\end{align*}
since the term including vertices outside of $B_{r}$ can be seen to be bounded.
\end{proof}

The case $\ell^{1}$ is more subtle in the general situation. In
\cite[Theorem~1.7]{HuK} it was shown that for the stochastic
complete graphs Yau's Liouville theorem remains true for $p=1$ (see
Section~\ref{s:SC} for definition of stochastic completeness).
Otherwise, there are counterexamples, see \cite[Section~4]{HuK}.

Refining the Caccioppoli inequality from the proof of
Theorem~\ref{t:Yau} above,  \cite[Lemma 3.1]{HuK}, and applying it
recursively, we obtain a discrete version of Karp's theorem.
However, for the recursive scheme we  need the jump size to be
finite (J).

\begin{thm}[Theorem~1.1 in \cite{HuK}]Let $b$ be a connected graph over $(X,m)$ and $\rho$ be an intrinsic metric with bounded degree on balls (D) and finite jump size (J). If $f$ is a positive subharmonic function such that for some $p\in (1,\infty)$ and $x\in X$
\begin{align*}
    \inf_{r_{0}>0}\int_{r_{0}}^{\infty} \frac{r}{{\|f1_{B_{r}(x)\|}}_{p}^{p}}dr=\infty,
\end{align*}
then $f$ is constant, where $1_{B}$ is the characteristic function of a set $B\subseteq X$.
\end{thm}

In particular, the theorem above implies the result of Hua/Jost \cite{HJ} and even implies that a harmonic function $f$ satisfying
\begin{align*}
    \limsup_{r\to\infty}\frac{1}{r^2\log r} \|f1_{B_{r}(x)}\|_{p}^{p}<\infty,
\end{align*}
for some $p\in(1,\infty)$, is constant.


\subsection{Recurrence}
As a direct consequence of Karp's theorem we get a sufficient criterion for recurrence  of a graph.
A connected graph $b$ over $X$ is called \emph{recurrent} if
for all $m$ and some (all) $x,y\in X$,
we have
$$    \int_{0}^{\infty}e^{-tL}1_{\{x\}}(y) dt=\infty$$
which is equivalent to absence of non-constant bounded subharmonic functions. For a collection of various equivalent statements of recurrence, see \cite[Proposition~3.3]{HuK} and references therein.

Similar  analogous results to the criterion below are due
to \cite[Theorem~3.5]{Ka} and \cite[Theorem~3]{Stu} which
generalizes for example \cite[Theorem~2.2]{DK},
\cite[Corollary~B]{RSV},
\cite[Lemma~3.12]{Woe}, \cite[Corollary~1.4]{Gri}, \cite[Theorem~1.2]{MUW}
on graphs.

\begin{thm}[Corollary 1.6 in \cite{HuK}] Let $b$ be a connected graph over $(X,m)$ and $\rho$ be an intrinsic metric with bounded degree on balls (D) and finite jump size (J). If for some  $x\in X$
\begin{align*}
    \int_{1}^{\infty} \frac{r}{m({B_{r}(x)})}dr=\infty,
\end{align*}
then the graph is recurrent.
\end{thm}


\subsection{Graphs with measure bounded from below}

Above we discussed existence of certain solutions in terms of metric properties of the underlying graph. There is a result for graphs of a completely different flavor which seems to have no analogue in the non-discrete setting. The  condition (A) below is about $(X,m)$ only as a measure space and the combinatorial structure of the graph.
\begin{itemize}
  \item [(A)] $\sum_{n=1}^{\infty}m(x_{n})=\infty$ for all infinite paths $(x_{n})$.
\end{itemize}
In particular, (A) is satisfied if
\begin{align*}
    \inf_{x\in X}m(x)>0
\end{align*}
which  for example holds if $m$ is constant as in the case of the counting measure or if $m=\deg$.

This condition yields a result for absence of certain solutions in $\ell^{p}(X,m)$ which in contrast to the metric result above includes the case $p=1$. Since the proof is rather short we include it here.

\begin{thm}[Lemma 3.2 in \cite{KL1}]\label{t:A} Let $(b,c)$ be a graph over $(X,m)$  such that  every infinite path has infinite measure  (A). Then there are no positive subharmonic function in $\ell^{p}(X,m)$, $p\in [1,\infty)$.
\end{thm}
\begin{proof} Let $f$ be non-negative and subharmonic.  Then, $\LL f(x)\leq0$ evaluated at some $x$ gives, using $f\ge0$,
\begin{align*}
 f(x)\leq\frac{1}{\sum_{y\in X}b(x,y)}\sum_{y\in X} b(x,y)f(y)
\end{align*}
Thus, whenever there is $x'\sim x$ with $f(x')<f(x)$ there must be $y\sim x$ such that $f(x)< f(y)$. Such $x'$ and $x$ exist whenever $f$ is non-constant. Letting $x_0=x$, $x_{1}=y$ and proceeding inductively there is a sequence $(x_{n})$ of vertices such that $0<f(x)<f(x_{n})< f(x_{n+1})$, $n\ge0$.
Now, (A) implies that $f$ is not in $\ell^{p}(X,m)$. On the other hand, if $f$ is constant, then (A) again  implies $f\equiv 0$.
\end{proof}

It is an open problem to unify Theorem~\ref{t:Yau} and Theorem~\ref{t:A}.


\section{Domain of the generators and essential selfadjointness}\label{s:ESA}

In this section we address the question of identifying the domain of the generators $L_{p}$. Classically, the special case $p=2$ received  particular attention. Going back to investigations of Friedrichs and von Neumann a classical question is whether a symmetric operator on a Hilbert space has a unique selfadjoint extension. This property is often studied under the name essential selfadjointness.

We say a symmetric operator on a dense subspace of a Hilbert space is \emph{essentially selfadjoint} if it has a unique selfadjoint extension. Moreover, we say the operator has a \emph{unique Markovian extension} if there is a unique selfadjoint extension such that the corresponding semigroup is Markovian. Clearly, essential selfadjointness implies uniqueness of Markovian extensions.

A sufficient criterion for essential selfadjointness for positive symmetric operators on $\ell^{2}$ is the absence of solutions for $\lm<0$ in $\ell^{2}$ to the equation $\mathcal{L}u=\lm u$, see \cite[Proof of Theorem~6]{KL1}.

The connection of essential selfadjointness to metric completeness for the Laplace Beltrami operator on Riemannian manifolds is that if there exists a boundary one might have to choose suitable 'boundary conditions' in order to obtain a selfadjoint operator.

We first discuss the manifold case which is often referred to Gaffney's theorem. Secondly, we consider weighted graph and recover Gaffney's theorem by the virtue of intrinsic metrics. Furthermore, we determine the domain of the generators on $\ell^{p}$.


\subsection{Gaffney's theorem for manifolds}

We first discuss the question for the Laplace Beltrami operator on a Riemannian manifold.

A  result going back to the work of Gaffney \cite{Ga0,Ga} essentially states that
on a geodesically complete manifold the so called Gaffney Laplacian is essentially selfadjoint which is equivalent to the uniqueness of Markovian extensions of the minimal Laplacian. Independently, essential selfadjointness of the Laplace Beltrami operator the compactly supported infinitely often differentiable functions was shown by Roelcke, \cite{R}. For later results in this direction see also \cite{Che, Str}.


\subsection{Domain of the generators for graphs and intrinsic metrics}
The first results connecting metric completeness and uniqueness of selfadjoint extensions were obtained by Colin de Verdi\`{e}re/Torki-Hamza/Truc \cite{CdVTHT} and Milatovic \cite{Mi,Mi2}. These results were proven for (magnetic) Schr\"odinger operators on graphs with  bounded combinatorial vertex degree and the metric $\de$ discussed in  Section~\ref{s:Metric:CdV}. As discussed there, $\de$ is equivalent to an intrinsic metric if and only if the combinatorial vertex degree is bounded. Here, we consider intrinsic metrics to deal with the general case of unbounded vertex degree.

In Lemma~\ref{l:DF:C_c} we demonstrated that for a graph  we may not
have $\LL C_{c}(X)\subseteq\ell^{2}(X,m)$. Hence, in general $\LL$
is not a symmetric operator on the subspace $C_{c}(X)$ of
$\ell^{2}(X,m)$. Nevertheless, we can still identify the domain of
the generator of the form. The following result is found in
\cite{HKMW} for graph Laplacians and in \cite{GKS} for magnetic
Schr\"odinger operators. The proof follows essentially from
Theorem~\ref{t:Yau} above and standard arguments found in
\cite[Proof of Theorem 5 and 6]{KL1}. One may find a version of the
theorem below in \cite[Corollary~1.4]{HuK}.

\begin{thm}Let $(b,c)$ be a graph over $(X,m)$ and $\rho$ be an intrinsic metric  with bounded degree on balls (D). Then,
\begin{align*}
    D(L_{p})=\{f\in \ell^{p}(X,m)\mid \LL f\in \ell^{p}(X,m)\},\quad\mbox{for all $p\in(1,\infty)$.}
\end{align*}
If additionally $\LL C_{c}(X)\subseteq\ell^{2}(X,m)$, then $\LL\vert_{C_{c}(X)}$ is essentially selfadjoint on $\ell^{2}(X,m)$.
\end{thm}

Combining this with the Hopf-Rinow type theorem, Theorem~\ref{t:HopfRinow}, we obtain an analogue to a classical result in Riemannian geometry which is often referred to as Gaffney's theorem.

\begin{corollary}[Theorem~2 in \cite{HKMW}] Let $b$ be a locally finite graph over $(X,m)$ and $\rho$ be an intrinsic path metric. If $(X,\rho)$ is metrically complete, then  $\LL\vert_{C_{c}(X)}$ is essentially selfadjoint on $\ell^{2}(X,m)$.
\end{corollary}


\subsection{Graphs with measure bounded from below}

Looking at the graph as a measure space we obtain the following
result as a consequence of Theorem~\ref{t:A}.

\begin{thm}[Theorem~5 in \cite{KL1}]  Let $(b,c)$ be a graph over $(X,m)$  such that  every infinite path has infinite measure  (A). Then,
\begin{align*}
    D(L_{p})=\{f\in \ell^{p}(X,m)\mid \LL f\in \ell^{p}(X,m)\}\quad\mbox{for all $p\in[1,\infty)$.}
\end{align*}
If additionally $\LL C_{c}(X)\subseteq\ell^{2}(X,m)$, then $\LL\vert_{C_{c}(X)}$ is essentially selfadjoint on $\ell^{2}(X,m)$.
\end{thm}


\section{Stochastic completeness}\label{s:SC}
In the sections above  we excluded the case of solutions in
$\ell^{\infty}$. Indeed, this case is more subtle and the topic of
this section. Specifically, we study  stochastic completeness, a
property which is  also referred to as conservativeness, honesty or
non-explosion.

There is a huge body of literature from various mathematical fields and for references we restrict ourselves a very small selection. For stochastic completeness in the context of discrete Markov processes there is work by Feller \cite{Fel, Fel2} and \cite{Reu} in  the late 50's, for manifolds there is work going back to Azencott \cite{Az} and Grigor'yan \cite{Gri86,Gri99}. For positive contraction semigroups we mention Arlotti/Banasiak \cite{AB} and Mokhtar-Kharroubi/Voigt \cite{M-KV} and for strongly local Dirichlet forms there is for example work by Sturm \cite{Stu}.

Below we discuss a volume growth bound for manifolds that implies stochastic completeness due to Grigor'yan  \cite{Gri86,Gri99}. Afterwards, we discuss stochastic completeness for graphs and show how Grigor'yan's bound fails when one considers the combinatorial graph metric. Finally, we present results that recover the result for manifolds for weighted graphs with intrinsic metrics due to Folz \cite{Fol} and Huang \cite{Hu1}.


\subsection{Grigor'yan's theorem for manifolds}
For a Riemannian manifold $M$ stochastic completeness is defined by the property that the diffusion semigroup of the Laplace-Beltrami operator leaves the constant function $1$ invariant. This can be seen to be equivalent to uniqueness of bounded solutions for the heat equation, \cite[Theorem~6.2]{Gri99}. In 1986 Grigor'yan \cite{Gri86} proved a theorem which gives a sufficient condition for stochastic completeness in terms of volume growth for connected Riemannian manifolds. In particular, this theorem states that if
\begin{align*}
    \inf_{r_{0}>0}\int_{r_{0}}^{\infty} \frac{r}{\log(\mathrm{vol}(B_{r}(x)))}dr=\infty,
\end{align*}
for some $x\in M$,
then the manifold is stochastically complete. In particular, this implies stochastic completeness for manifolds whose volume growth is less than $e^{r^{2}}$. Grigor'yan's theorem was later generalized to strongly local regular Dirichlet forms by Sturm \cite{Stu}.


\subsection{Characterization of stochastic completeness for graphs}
Next, we discuss the notion of stochastic completeness for graphs in more detail. We start with a characterization linking stochastic completeness to solutions in $\ell^{\infty}$.

Variants of the next proposition can be found in \cite{Fel,Fel2,Reu} in the context of discrete Markov processes, in \cite{Woj1} for graphs with standard weights and in \cite[Theorem~1]{KL1} for weighted graphs (which is given there even for non vanishing killing term $c$).

\begin{proposition} Let $b$ be a connected graph over $(X,m)$. The following are equivalent:
\begin{itemize}
  \item [(i)] $e^{-tL}1=1$ for some (all) $t>0$.
  \item [(ii)] For any (some) $\lm<0$,  there is no non-trivial $u\in \ell^{\infty}(X) $ such that
      \begin{align*}
        \LL u=\lm u.
      \end{align*}
 \item [(iii)] For any (some) $f\in \ell^{\infty}(X)$ there is a unique  solution $ [0,\infty)\to\ell^{\infty}(X)$, $t\mapsto u_{t}$  to
      \begin{align*}
        -\LL u=\partial_{t}u,\quad u_{0}=f.
      \end{align*}
\end{itemize}
\end{proposition}
We call a graph $b$ over $(X,m)$ \emph{stochastically complete} if one of the equivalences of the proposition above is satisfied.

There is a physical interpretation to (i) of the above theorem. This
concerns the question whether heat leaves the graph in finite time.
Assume the graph is not stochastically complete, i.e., $e^{-tL}1< 1$
for some $t>0$ (recall that we always have $e^{-tL}1\le 1$
  since $e^{-tL}$ is Markovian). Let $0\leq f\in \ell^{1}(X,m)$ model the
distribution of heat in the graph at time $t=0$. Then, the
distribution of heat at time $t>0$ is given by $e^{-tL}f$ and the
amount of heat in the graph at time $t>0$ is given by
\begin{align*}
    \sum_{x\in X}e^{-t L} f(x)m(x)=\langle e^{-tL}f,1\rangle =\langle f,e^{-tL}1\rangle < \langle f,1\rangle = \sum_{x\in X} f(x)m(x),
\end{align*}
where the right hand side is the amount of heat in the graph at time $t=0$. Hence, the amount of heat in the graph decreases in time, if the graph is not stochastically complete.

\subsection{The combinatorial graph distance and polynomial growth}\label{s:SC_graphs}

Consider the combinatorial graph distance on a graph with standard weights with the counting measure. Wojciechowski discovered in his PhD Thesis \cite{Woj1} that spherically symmetric trees, whose branching numbers grow more than linearly,  are not stochastically complete. This counts for a volume growth of $r!\sim e^{r\log r}$ with respect to the combinatorial graph distance. Later, Wojciechowski \cite[Example 4.11]{Woj3}  gave even examples of stochastically complete graphs with polynomial volume growth. These considerations were generalized for weighted weakly spherically symmetric graphs in \cite{KLW}.

The examples Wojciechowski considered are so called
\emph{anti-trees}. Specifically, anti-trees are highly connected
graphs that are characterized by the property that a vertex in a
sphere (with respect to a root vertex) is connected to every
neighbor in the succeeding sphere, see  Figure~\ref{f:antitree}
below for an example.

\begin{center}
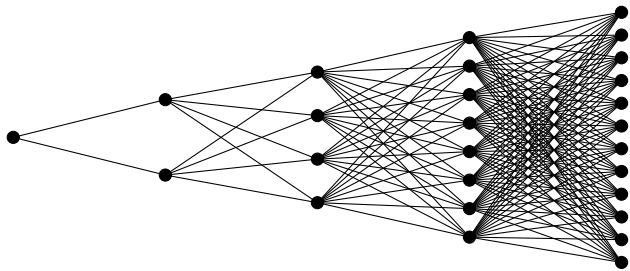
\begin{figure}[h!]
 \begin{tikzpicture}[scale=1]
  \filldraw (-1*2,{0}) circle (0.8mm);
     \foreach \Y in {0,...,1} {
  \draw (-1*2,{0}) -- (0,{(\Y-1/2)/sqrt(1-0)});
        }
  \def\xab{0/0/1/-1/2 , 1/-1/2/-3/4 , 2/-3/4/-5/6 , 3/-5/6/1/0 }
  \foreach \x/\a/\b/\A/\B in \xab {
    \foreach \y in {\a,...,\b} {
      \filldraw (2*\x,{(\y-1/2)/sqrt(\b-\a)}) circle (0.8mm);
      \ifthenelse{\A>\B}{}{
        \foreach \Y in {\A,...,\B} {
          \draw (2*\x,{(\y-1/2)/sqrt(\b-\a)}) -- (2*\x+2,{(\Y-1/2)/sqrt(\B-\A)});
        }
      }
    }
  }
  \end{tikzpicture}
  \caption{An anti-tree with $s_{r+1}=2^{r}$}\label{f:antitree}
\end{figure}
\end{center}
So,  for an anti-tree  let $s_{r}$ be the number of vertices with
combinatorial graph distance  $r$ to a root vertex and let
$v_r=s_{0}+\ldots +s_{r}$, $r\ge0$. Then, Wojciechowski shows that
the anti-tree with standard weights  and the counting
  measure is stochastically complete if
and only if
\begin{align*}
    \sum_{r=0 }^{\infty}\frac{v_{r}}{s_{r}s_{r+1}}=\infty
\end{align*}
which results in a threshold of $v_{r}\sim r^{3}$. Hence, there are graphs with only little more than cubic volume growth that are not stochastically complete. This stands in clear contrast to Grigor'yan's volume growth bound for manifolds which was about $e^{r^{2}}$. Thus, the example clearly shows that the combinatorial graph distance is the wrong candidate to obtain an analogue of Grigor'yan's result.

Moreover, it was shown later in \cite{GHM} that the threshold of $r^{3}$ is sharp for the combinatorial graph distance.

\subsection{Grigor'yan's result for graphs with intrinsic metrics}

The following theorem was proven by Folz \cite{Fol} and later a
simplified proof was given by Huang \cite{Hu1}. Before, a first step in
this direction was taken by Masamune/Uemura \cite{MU} and Grigor'yan/Huang/Masamune \cite{GHM}.

\begin{thm}[Theorem~1 of \cite{Fol}] Let $b$ be a locally finite, connected graph over $(X,m)$ and $\rho$ be an intrinsic metric  with finite balls (B) and finite jump size (J). If for some $x\in X$
\begin{align*}
    \inf_{r_{0}>0}\int_{r_{0}}^{\infty} \frac{r}{\log(m(B_{r}(x)))}dr=\infty,
\end{align*}
then the graph is stochastically complete.
\end{thm}

The assumption of finite jump size and the assumption that the metric $\rho$ is intrinsic were weakened by Huang \cite{Hu1} to so called
weakly adapted metrics that is a metric $\rho$ for which there is $c>0$ such that
$$\sum_{y\in X}b(x,y)(\rho(x,y)\wedge c)^{2}\leq m(x),\quad x\in X.$$

The approach of Folz uses quantum graphs. These are graphs were the
edges are real intervals which are glued at the vertices. For some
background see Kuchment \cite{Ku}. Barlow/Bass \cite{BB}  already
used a similar idea to study  cable systems. The corresponding
Laplace operator acts as  $-\tfrac{d^{2}}{dx^{2}}$ on the intervals
and certain boundary conditions are imposed such that it becomes a
selfadjoint operator associated to a Dirichlet form. This Dirichlet
form is strongly local and, thus, Grigor'yan's result holds by the
virtue of \cite{Stu}. In order to compare the diffusion on the
quantum graph and the discrete graph one has to add a certain number
of loops to each vertex \cite{Fol}. Alternatively, one can model the
time change by appropriate Dirac measures on the vertices
\cite{Hu1}.


\section{Upper escape rate}\label{s:UER}
Khintchine's law of the iterated logarithm
states for the Brownian motion $(B_{t})_{t\ge0}$ on $\R$,
\begin{align*}
\limsup_{t\to\infty}
\frac{|B_t|}{\sqrt{ 2t \log \log t}}
= 1,
\end{align*}
almost surely.
In particular,  for the probability conditioned on the Brownian motion starting at zero, we have
\begin{align*}
    \PP_{0} \Big( |B_t| \leq R(t) =\sqrt{(2 + \eps)t \log \log t} \mbox{ for all sufficiently large $t$ }\Big) = 1.
\end{align*}
for arbitrary $\eps>0$. Such a function $R$ is called an \emph{upper escape rate function}. Below we discuss first results on manifolds in this direction by Grigor'yan, Hsu and Qin before we get to weighted graphs and intrinsic metrics.


\subsection{Upper escape rate functions for manifolds}
For a complete connected Riemannian manifold $M$ which satisfies
\begin{align*}
    \inf_{r_{0}>0}\int_{r_{0}}^{\infty} \frac{r}{\log(\mathrm{vol}(B_{r}(x)))}dr=\infty,
\end{align*}
for some $x\in M$, we have seen in the previous section that $M$ is stochastically complete. Imposing this assumption the result of  Hsu/Qin \cite{HQ} gives for the Brownian motion $(B_{t})_{t\ge0}$ on $M$ and the probability $\PP_{x}$ conditioned on the process to start in $x\in M$,
\begin{align*}
     \PP_{x} \Big( |B_t| \leq R(t) \mbox{ for all sufficiently large $t$ }\Big) = 1,
\end{align*}
where the upper escape rate function $R$ is a multiple of the inverse function of
\begin{align*}
    t\mapsto \int_{6}^{t}\frac{r}{\log\mathrm{vol}(B_{r})+\log\log r}dr.
\end{align*}
Earlier results in this direction are found in \cite{Gri99b} and \cite{GH09}. For strongly local Dirichlet forms this result was recently generalized by Ouyang, \cite{O}.

\subsection{Upper escape rate functions for weighted graphs}
For a graph $b$ over $(X,m)$, let $(X_{t})$ be  Markov process with generator $L$. The process is characterized by the   formula
\begin{align*}
    e^{-tL}f(x)=\EE_{x}(f(X_{t})),\qquad t>0,\, f\in C_{c}(X).
\end{align*}
The following theorem on the upper escape rate is due to
Huang/Shiozawa \cite{HS} which improves the results of Huang
\cite{Hu3}.

\begin{thm}[Theorem~1.7 of \cite{HS}] Let $b$ be a locally finite connected graph over $(X,m)$ such that $\inf_{x\in X}m(x)>0$ and $\rho$ be an intrinsic metric  with finite balls (B) and finite jump size (J). If for some $x\in X$
\begin{align*}
    \inf_{r_{0}>0}\int_{r_{0}}^{\infty} \frac{r}{\log(m(B_{r}(x)))}dr=\infty,
\end{align*}
then
\begin{align*}
     \PP_{x} \Big( |X_t| \leq R(t) \mbox{ for all sufficiently large $t$ }\Big) = 1,
\end{align*}
where the upper escape rate function $R$ is a multiple of the inverse function of
\begin{align*}
    t\mapsto \int_{t_0}^{t}\frac{r}{\log\mathrm{vol}(B_{r})+\log\log r}dr.
\end{align*}
for some  $t_{0}$.
\end{thm}

Note that the proof found in \cite{HS} also gives an alternative argument to show stochastic completeness under the volume growth condition above.

It can be shown that the escape rate is sharp for anti-trees. Moreover, for weakly spherically symmetric graphs there are even more precise results, confer \cite[Section~6]{Hu3}.


\section{Subexponentially bounded solutions and the spectrum}\label{s:Shnol}
So far, we considered solutions for $\lm<0$. In this section we turn to solutions for $\lm\ge0$ to study  the spectrum $\si(L)$ of the operator $L$ which is included in $[0,\infty)$ since $L$ is positive.

The meta-theorem behind this is often referred to as a Shnol' type
theorem and it states that $\lm$ is in the spectrum if there is a
subexponentially bounded solution for $\lm$.  The classical Shnol'
theorem \cite{Sh} deals with the Schr\"odinger equation in $\R^{d}$.
It was rediscovered by Simon in \cite{Si1,Si2} and it was proven in
\cite{BLS} for strongly local Dirichlet forms by Boutet de
Monvel/Lenz/Stollmann.

A function $f\in C(X)$ on a connected graph is said to be \emph{subexponentially bounded} with respect to an intrinsic metric $\rho$ if for all $\al>0$ and some (all) $x\in X$
\begin{align*}
    e^{-\al\rho(x,\cdot)}f\in \ell^{2}(X,m).
\end{align*}

For bounded Laplacians on graphs a Shnol' type theorem was proven in \cite{HK}. Using intrinsic metrics we get a general version of this result. Such a result can be derived  from a  Shnol' inequality \cite[Corollary 12.2]{FLW} combined with a one-dimensional estimate. As it is stated here the theorem can be found in \cite[Theorem~19]{K3}.

\begin{thm}[Corollary 12.2 in \cite{FLW}]Let $(b,c)$ be a connected graph over $(X,m)$ and $\rho$ be an intrinsic metric  with finite balls (B) and finite jump size (J). If for some $\lm\in \R$ there is a non-trivial subexponentially bounded solution, then $\lm\in\si(L)$.
\end{thm}
The proof is based again on a Caccioppoli type inequality.

A basic corollary which we improve later in Section~\ref{s:Brooks} is the following.

\begin{corollary}Let $b$ be a  connected graph over $(X,m)$ and $\rho$ be an intrinsic metric  with finite balls (B) and finite jump size (J). If
\begin{align*}
    \limsup_{r\to\infty}\frac{1}{r}\log m(B_{r}(x))\leq 0,
\end{align*}
for some $x\in X$, then $\inf\si(L)=0$.
\end{corollary}
\begin{proof}Under the assumption above the constant functions which are solutions for $\lm=0$ are subexponentially bounded. Hence, the statement follows from the theorem above.
\end{proof}


\section{Isoperimetric constants and lower spectral bounds}\label{s:Cheeger}
In this section we aim for lower bounds on the bottom of the
spectrum
\begin{align*}
    \lm_{0}(L)=\inf \si(L)
\end{align*}
via so called isoperimetric estimates. Such estimates are often
referred to as Cheeger's inequality.

We first discuss the result on manifolds going back to Cheeger from
1960. Then, we discuss how an analogous  result was proven in the
80's for the normalized Laplacian by Dodziuk/Kendall and what
problems occur for the operator $\Delta$. Finally, we examine how
intrinsic metrics can be used to establish this inequality for
general graph Laplacians.


\subsection{Cheeger estimates for manifolds}
For a non-compact  Riemannian manifold $M$ the isoperimetric
constant or Cheeger constant is defined as
\begin{align*}
    h_{M}=\inf_{S}\frac{\mathrm{Area(\partial S)}}{\mathrm{vol}(\mathrm{int}(S))},
\end{align*}
where  $S$ runs  over all hypersurfaces cutting $M$ into a
precompact piece $\mathrm{int}(S)$ and an unbounded piece. Denote by
$\lm_{0}(\Delta_{M})$  the bottom of the spectrum of the
Laplace-Beltrami. The well known Cheeger inequality reads as
\begin{align*}
    \lm_{0}(\Delta_{M})\ge\frac{h_{M}^{2}}{4}.
\end{align*}
See \cite{Ch} for Cheeger's original work on the compact case and \cite{Br2} for a discussion of the non-compact case.

\subsection{Cheeger estimates for graphs with standard weights}
There is an enormous  amount of  literature on isoperimetric
inequalities especially for finite graphs. Here, we restrict
ourselves to infinite graphs and only mention \cite{AM} for finite
graphs.

The \emph{boundary} of a set $W\subseteq X$ is defined as  the set of edges emanating from $W$, i.e.,
\begin{align*}
    \partial W=\{(x,y)\in W\times X\setminus W\mid x\sim y\}
\end{align*}
In 1984  Dodziuk, \cite{Do}, considered graphs with standard weights and the counting measure.  The isoperimetric constant he studied  is closely related to
\begin{align*}
h_{1}=\inf_{\mbox{\scriptsize{$W\subseteq X$ finite}}} \frac{|\partial W|}{|W|}
\end{align*}
and  Dodziuk's proof yields
\begin{align*}
    \lm_{0}(\Delta)\ge \frac{h_{1}^{2}}{2D},
\end{align*}
with $D=\sup_{x\in X}\deg(x)$. This analogue of Cheeger's inequality
is useful for graphs with bounded vertex degree. However, for
unbounded vertex degree the bound becomes trivial. The following
example illustrates how this bound can be seen to be non optimal.

\begin{example}Let $T_{k}$, $k\ge2$, be the $k$-regular rooted tree with standard weights. We glue $T_{k}$ to $T_{k^{4}}$  at the root and denote the resulting graph by $G_{k}$. We find that the infimum of the bottom of the spectrum of $\Delta_{G_{k}}$ on $G_{k}$ is assumed at $\lm_{0}(\Delta_{G_{k}})=\lm_{0}(\Delta_{T_{k}})=((k+1)-2\sqrt{k})$. In particular, $\lm_{0}(\Delta_{G_{k}})\to\infty$ for $k\to\infty$. The constant $h_{1}({G_{k}})$ is assumed by considering larger and larger balls on $T_{k}$, i.e., $h_{1}({G_{k}})=h_{1}({T_{k}})=k$. Finally, $D=k^{4}$. Hence, the bound in Dodziuk's inequality is $k^2/2(k+1)^{4}\to0$, $k\to\infty$. In summary, $\lm_{0}\to\infty$ while $h_{1}^{2}/2D\to0$ as $k\to\infty$.
\end{example}

Two years later Dodziuk and Kendall  \cite{DK} proposed a solution to this issue by considering graphs with standard weights and the normalizing measure $n=\deg$ instead. The corresponding isoperimetric constant is
\begin{align*}
h_{n}=\inf_{\mbox{\scriptsize{$W\subseteq X$ finite}}} \frac{|\partial W|}{\deg(W)}
\end{align*}
and they proved in \cite{DK} for the normalized Laplacian $\Delta_{n}$
\begin{align*}
    \lm_{0}(\Delta_{n})\ge \frac{h_{n}^{2}}{2}.
\end{align*}
This analogue of Cheeger's inequality does not have the disadvantage illustrated in the example above. It seems that from this point on the operator $\Delta$ was rather neglected in the spectral geometry of graphs and the normalized Laplacian $\Delta_{n}$ gained momentum.

This estimate can be  generalized to  arbitrary edge weights $b$. Define the area of the boundary of a finite set $W\subseteq X$ by
\begin{align*}
    b(\partial W)=\sum_{(x,y)\in\partial W}b(x,y)
\end{align*}
and for the normalizing measure $n$  with $n(W)=\sum_{x,y\in
W}b(x,y)$ the normalized Cheeger constant is given by
\begin{align*}
\al_{n}=\inf_{\mbox{\scriptsize{$W\subseteq X$ finite}}} \frac{b(\partial W)}{n(W)}.
\end{align*}
With further improvements to Dodziuk/Kendall's bound found in \cite{Fuj2,KL2} one has for the operator $L$ associated to the graph $b$ over $(X,n)$
\begin{align*}
 \lm_{0}(L)\ge 1-\sqrt{1-{\al_{n}^{2}}},
\end{align*}
where the left hand side can be seen to be larger than ${\al_{n}^{2}}/{2}$ by the Taylor expansion of the square root.


\subsection{Cheeger estimates involving intrinsic
metrics}\label{s:Cheeger:intrinsic} The considerations in the
previous sections above suggest that intrinsic metrics allow to
prove the analogous  results for general graph Laplacians. So, the
question is where the metric actually appears in the definition of
the isoperimetric constant.

Revisiting the definition of the area of the boundary of  the set $W$ above we find that
\begin{align*}
    b(\partial W)=\sum_{(x,y)\in\partial W}b(x,y)=\sum_{(x,y)\in\partial W}b(x,y)d(x,y)
\end{align*}
with the combinatorial graph distance $d$ on the right hand side. Remember that $d$ is an intrinsic metric for the graph $b$ over $(X,n)$.

Hence, replacing $d$ by an intrinsic metric $\rho$ for a graph $b$ over $(X,m)$ we define
\begin{align*}
    \mathrm{Area}(\partial W)=\sum_{(x,y)\in\partial W}b(x,y)\rho(x,y).
\end{align*}
That is we take the length of an edge into consideration to measure the area of the boundary.
We define
\begin{align*}
    \al=\inf_{\mbox{\scriptsize{$W\subseteq X$ finite}}}\frac{\mathrm{Area}(\partial W)}{m(W)},
\end{align*}
and  obtain the following theorem which is  found in \cite{BKW}.

\begin{thm}[Theorem~1 in \cite{BKW}] Let $b$ be a graph over $(X,m)$ and let $\rho$ be an intrinsic metric. Then,
\begin{align*}
    \lm_{0}(L)\ge \frac{\al^{2}}{2}.
\end{align*}
\end{thm}
The proof of the theorem is based on an area and a co-area formula. For $f\ge0$ let
\begin{align*}
    \Om_{t}=\{x\in X\mid f(x)>t\}.
\end{align*}
Then one can prove using   Fubini's theorem for $f\in C_{c}(X)$
\begin{align*}
m(\Om_{t})&=\sum_{x\in X}f(x)m(x)\\
\mathrm{Area}(\partial\Omega_{t})&=\sum_{x,y\in X}b(x,y)\rho(x,y)|f(x)-f(y)|.
\end{align*}
The proof is then basically Cauchy-Schwarz inequality and various algebraic manipulations.

One may also include potentials $c\ge0$ in the estimate by introducing edges from  vertices $x$ with $c(x)>0$ to virtual sibling vertices $\dot x$ with edge weight $b(x,\dot x)=c(x)$. The union of vertices $x\in X$ and $\dot x$ is denoted by $\dot X$. Furthermore, we extend an intrinsic metric $\rho$ on $X$ to the new edges via
\begin{align*}
    \rho(x,\dot x)=\frac{(m(x)-\sum_{y\in X}b(x,y)\rho(x,y)^{2})^{\frac{1}{2}}}{c(x)}.
\end{align*}
The extension of $\rho$ becomes an intrinsic metric when choosing $m(\dot x)=m(x)$. Now, we define $\al$ by taking the infimum of the quotient with the extension of $b$ and $\rho$ but as above only over subsets of $X$.


\section{Exponential volume growth and upper spectral bounds}\label{s:Brooks}
In this section, we discuss upper bounds for the bottom of the essential spectrum
\begin{align*}
    \lm_{0}^{\mathrm{ess}}(L)=\inf \si_{\mathrm{ess}}(L).
\end{align*}
The essential spectrum of an operator is the part of the spectrum which does not include discrete eigenvalues of finite multiplicity. Clearly, $\lm_{0}(L)\leq\lm_{0}^{\mathrm{ess}}(L)$.

We discuss the classical result on Riemannian manifolds going back to Brooks first. Then we present a corresponding result for the normalized Laplacian and show how the result fails in the case of Laplacian with respect to the counting measure. Finally, we employ intrinsic metrics to recover Brooks' result for general graph Laplacians.


\subsection{Brooks' theorem for manifolds}
Let $M$ be a complete connected non-compact Riemannian manifold with infinite volume. Let $\lm_{0}^{\mathrm{ess}}(\Delta_{M})$ be the bottom of the essential spectrum of the Laplace Beltrami operator $\Delta_{M}$. Let $\overline{\mu}_{M}$ be the upper exponential growth rate of the distance balls
\begin{align*}
    \overline{\mu}_{M}=\limsup_{r\to\infty} \frac{1}{r}\log\mathrm{vol}(B_{r}(x)),
\end{align*}
for an arbitrary $x\in M$.  Brooks showed in 1981, \cite{Br},
\begin{align*}
    \lm_{0}^{\mathrm{ess}}(\Delta_{M})\leq \frac{\overline{\mu}_{M}^{2}}{4}.
\end{align*}
Later in 1996 Sturm, \cite{Stu}, showed using the lower exponential growth rate of the distance balls with variable center
\begin{align*}
     \underline{\mu}_{M} =\liminf_{r\to\infty}\inf_{x\in M} \frac{1}{r}\log\mathrm{vol}(B_{r}(x))
\end{align*}
the following bound
\begin{align*}
    \lm_{0}(\Delta_{M})\leq \frac{\underline{\mu}_{M}^{2}}{4}.
\end{align*}
Indeed, this result in  \cite{Stu} is  proven for strongly local regular Dirichlet forms.

An immediate corollary is that $0$ is in the spectrum of the Laplace Beltrami operator for $M$ with subexponential growth, i.e., $\underline{\mu}_{M}=0$.

\subsection{Brooks' theorem for graphs with standard weights}

For graphs with standard weights and the normalizing measure Dodziuk/ Karp \cite{DKa} proved in 1987 the first analogue of Brooks' theorem for graphs. This result was later improved by Ohno/Urakawa \cite{OU} and Fujiwara \cite{Fuj1} resulting in the estimate
\begin{align*}
        \lm_{0}^{\mathrm{ess}}(\Delta_{n}) \leq 1-\frac{2e^{\mu_{n}/2}}{e^{\mu_{n}}+1}
\end{align*}
with
\begin{align*}
    \mu_{n}=\limsup_{r\to\infty}\frac{1}{r}\log n(B_{r}(x)),
\end{align*}
for arbitrary $x\in X$ and $n=\deg$. It can be seen that the bound above is smaller than $\mu_{n}^{2}/8$.

Next, we discuss how for graphs with standard weights and the counting measure such a bound fails when volume growth is determined by the combinatorial graph distance.  Again, we consider anti-trees which were introduced in Section~\ref{s:SC_graphs}. In \cite{KLW} it was shown that \begin{align*}
    a=\Big(\sum_{r=0}^{\infty}\frac{v_{r}}{s_{r}s_{r+1}}\Big)^{-1}
\end{align*}
is a lower bound on the spectrum of $\Delta$ on the anti-tree with
$s_{r}$ vertices in the $r$-th sphere and $v_r=s_{0}+\ldots +s_{r}$,
$r\ge0$, (where $a=0$ if the sum diverges). Moreover, in the case
where the sum converges the spectrum of $\Delta$ is purely discrete,
i.e., there is no essential spectrum. In particular, this implies
that anti-trees with $$s_{r}\sim r^{2+\eps},\qquad\eps>0,$$ have
positive bottom of the spectrum and no essential spectrum. However,
for $s_{r}\sim r^{2+\eps}$ we have $v_{r}\sim r^{3+\eps}$. That is
these are graphs of little more than cubic growth but these graphs
have positive bottom of the spectrum and no essential spectrum.
Hence, there is no analogue to Brooks' or Sturm's theorem for
$\Delta$ with respect to the combinatorial graph distance.

\subsection{Brooks' theorem involving intrinsic metrics}
Let $b$ be a graph over $(X,m)$ and $\rho$ be an intrinsic metric. Let
$B_{r}(x)$ be the distance $r$ ball about a vertex $x$ with respect to
the metric $\rho$. We define
\begin{align*}
    \mu=\liminf_{r\to\infty}\frac{1}{r}\log m(B_{r}(x)),
\end{align*}
for fixed $x\in X$ and
\begin{align*}
    \underline{\mu}=\liminf_{r\to\infty}\inf_{x\in X}\frac{1}{r}\log
    m(B_{r}(x)).
\end{align*}

In \cite{HKW}  analogues of Brooks' and Sturm's theorem for regular Dirichlet forms are proven. As a special case the following theorem is obtained. Under somewhat stronger assumptions the estimate for the essential spectrum was also obtained independently by Folz \cite{Fol2}.

\begin{thm}[Corollary 4.2 in \cite{HKW}] Let $b$ be a connected graph over $(X,m)$ and
$\rho$ be an intrinsic metric such that the balls are finite (B).
Then,
\begin{align*}
    \lm_{0}(L)\le \frac{\underline{\mu}^{2}}{8}.
\end{align*}
If furthermore $m(X)=\infty$, then
\begin{align*}
    \lm_{0}^{\mathrm{ess}}(L)\le \frac{\mu^{2}}{8}.
\end{align*}
\end{thm}

We indicate the idea of the proof.
\begin{proof}[Idea of the proof]
Let $\ov
\mu=\limsup_{r\to\infty}\frac{1}{r}\log m(B_{r}(x))$. Then, the
functions $f_{a}=e^{-a\rho(o,\cdot)}$ for $a>\ov{\mu}/2$ and
fixed $o$ are in $\ell^{2}(X,m)$. Moreover, by the mean value
theorem and an estimate as in Section~\ref{s:Metric:estimate}
we find that
\begin{align*}
    \QQ(f_{a})\leq \frac{a^{2}}{2}\sum_{x\in X}|f_a(x)|^{2}\sum_{y\in
    X} b(x,y)\rho(x,y)^{2}\leq\frac{a^{2}}{2}\|f_{a}\|^{2}
\end{align*}
To pass from $\ov{\mu}$ to $\mu$ or $\underline{\mu}$ we
consider
$$g_{a,r}=(e^{2a r}f_{a}-1)\vee 0.$$
Note that $g_{a,r}$ is supported on $B_{2r}$ and, therefore, $g_{a,r}$ is in
$C_{c}(X)$ whenever (B) applies.  Finally, to see the statement for the essential spectrum we
need to modify $g_{a,r}$ such that we obtain a sequence of
functions that converge weakly to zero. We achieve this by cutting
off $g_{a,r}$ at $1$ on $B_{r}$, i.e.,
\begin{align*}
    h_{a,r}=1\wedge g_{a, r}.
\end{align*}
The weak convergence of $h_{a,r}$ to zero is ensured by the
assumption $m(X)=\infty$. Now, the statement follows by a
Persson-type theorem, \cite[Proposition~2.1]{HKW}.
\end{proof}

Let us end this section with a few remarks.

In \cite{HKW} it is also shown that the assumption (B) can be replaced by (A).

As a corollary we get under the assumption of the theorem $2\al\leq
\mu$ for the Cheeger constant $\al$ of
Section~\ref{s:Cheeger:intrinsic}.

By comparing the degree path metric $\rho_{0}$ with the
combinatorial graph distance $d$ on anti-trees one finds that for
$s_{r}\sim r^{2-\eps}$ the balls with respect to $\rho_{0}$ grow
polynomially, for $s_{r}\sim r^{2}$ they grow exponentially and for
$s_{r}\sim r^{2+\eps}$ the graph has finite diameter with respect to
$\rho_{0}$. This shows that the examples in the section above are
indeed sharp.

\section{Uniform subexponential growth and $p$-independence of the
spectrum} 

In the beginning of the 80's Simon \cite{Si2} asked a famous question whether the spectra of
certain Schr\"odinger operators on $\R^{d}$ are independent on which
$L^{p}$ space they are considered. Hempel/Voigt \cite{HV} gave an affirmative answer in 1986. Here, we consider a geometric analogue of this question. First, we discuss a theorem by Sturm on Riemannian manifolds, \cite{Stu2}, and secondly, we present a result for weighted graphs involving intrinsic metrics.

\subsection{$p$-independence for  manifolds}\label{s:p_ind}
In 1993 Sturm \cite{Stu2} proved a theorem for uniformly elliptic operators on a
complete Riemannian manifold $M$ whose Ricci curvature is bounded
below. We assume that $M$ grows uniformly subexponentially if for
any $\eps>0$ there is $C>0$ such that for all $r>0$ and all $x\in M$
\begin{align*}
\mathrm{vol}(B_{r}(x))\leq C e^{\eps r}\mathrm{vol}(B_{1}(x)).
\end{align*}
Then the spectrum of a uniformly elliptic operator on this manifold is independent of the  space $L^{p}(M)$, $p\in[1,\infty]$ on which it is considered.

\subsection{$p$-independence for weighted graphs with intrinsic metrics}

A graph $(b,c)$ over $(X,m)$ with an intrinsic metric $\rho$ is said
to have \emph{uniform subexponential growth} if  for any $\eps>0$
there is $C>0$ such that for all $r>0$ and all $x\in M$
\begin{align*}
m(B_{r}(x))\leq C e^{\eps r}m(x).
\end{align*}
The proof of the following theorem follows closely the strategy of Sturm in \cite{Stu2} which makes it necessary to consider a complexification of the function spaces.

\begin{thm}[Theorem 1 in \cite{BHK}]
Let $b$ be a connected graph over $(X,m)$ and $\rho$ be an intrinsic
metric such that the balls are finite (B), which has finite jump
size (J) and the graph grows uniformly subexponentially. Then,
$$\si(L_{p})=\si(L_{2}),\qquad p\in[1,\infty].$$
\end{thm}

The statement of the theorem is in general wrong if one drops the
growth assumption. In particular, if $\lm_{0}(L_{2})>0$, then the
graph grows exponentially, i.e., $\mu>0$ by the section above. On
the other hand, if the graph is stochastically complete, (i.e., $e^{-tL}1=1$, see Section~\ref{s:SC_graphs}), then $1$ is an
eigenfunction of $L_{\infty}$ to the eigenvalue $0$ and, therefore,
$\lm_{0}(L_{1})=0$ by duality. Hence,
$\lm_{0}(L_{1})<\lm_{0}(L_{2})$.

On the other hand, it is an open question what happens for graphs
that are subexponentially growing, i.e., $\mu=0$, but not uniformly
subexponentially growing.
\medskip

\textbf{Acknowledgement}. The author  expresses his gratitude to Daniel Lenz for introducing him to the guiding theme of this survey that intrinsic metrics are the answer to many questions. Furthermore, several inspiring discussions with Xueping Huang resulted in various insights that improved this article. Moreover, the author appreciates the support and the kind invitation to the conference on 'Mathematical Technology of Networks - QGraphs 2013' at the ZiF in Bielefeld by Delio Mugnolo. Finally, funding by  the DFG is acknowledged.

\end{document}